\documentclass{article}
\usepackage{amssymb}
\usepackage{amsmath}
\usepackage{amsfonts}

\newtheorem{theorem}{Theorem}
\newtheorem{acknowledgement}[theorem]{Acknowledgement}

\newtheorem{corollary}[theorem]{Corollary}

\newtheorem{definition}[theorem]{Definition}

\newtheorem{lemma}[theorem]{Lemma}

\newtheorem{remark}[theorem]{Remark}

\newenvironment{proof}[1][Proof]{\noindent\textbf{#1.} }{\ \rule{0.5em}{0.5em}}
\input{tcilatex}

\begin{document}

\begin{center}
\textbf{{\Large {INTERPOLATION FUNCTION OF THE GENOCCHI \textbf{TYPE }%
POLYNOMIALS}}}

\bigskip

{\Large %\textit{\Large{Canada}}
}

Burak Kurt and Yilmaz Simsek

Akdeniz University, Faculty of Arts and Science, Department of Mathematics,
07058-Antalya, Turkey, burakkurt@akdeniz.edu.tr and ysimsek@akdeniz.edu.tr

\bigskip
\end{center}

\textbf{{\Large {Abstract }}}The main purpose of this paper is to construct
not only generating functions of the new approach Genocchi type\ numbers and
polynomials but also interpolation function of these numbers and polynomials
which are related to $a$, $b$, $c$ arbitrary positive real parameters. We
prove multiplication theorem of these polynomials. Furthermore, we give some
identities and applications associated with these numbers, polynomials and
their interpolation functions.

\bigskip



\section{Introduction, Definitions and Notations}

The history of the Euler numbers and the Genocchi numbers go beck to Euler
on16\textit{th} century and Genocchi on 19\textit{th} century, respectively.
From Euler and Genocchi to this time, these numbers can be defined in many
other ways. These numbers and polynomials play an important role in many
branch of Mathematics, for instance, Number Theory, Finite differences.
Therefore, applications of these numbers and their generating functions have
been investigated by many authors in the literature. Many kind of functions
are used to obtain generating functions of the Euler numbers and the
Genocchi numbers cf. (\cite{cangulozdensimsekAMH}-\cite{ysimsekKurtNaciKim}).

The classical Euler numbers $E_{n}$ are defined by means of the following
generating function.%
\begin{equation*}
\frac{2}{e^{t}+1}=\sum_{n=0}^{\infty }E_{n}\frac{t^{n}}{n!}\text{, }%
\left\vert t\right\vert <\pi \text{,}
\end{equation*}
cf. (\cite{cangulozdensimsekAMH}-\cite{ysimsekKurtNaciKim}).

The classics Genocchi numbers $G_{n}$ are defined by means of the following
generating function%
\begin{equation}
f_{G}(t)=\frac{2t}{e^{t}+1}=\sum_{n=0}^{\infty }G_{n}\frac{t^{n}}{n!}\text{, 
}\left\vert t\right\vert <\pi .  \label{1}
\end{equation}

The Genocchi numbers, named after Angelo Genocchi, are a sequence of
integers. This numbers are satisfies the following relations. By the \textit{%
umbral calculus }convention in (\ref{1}), we have the recurrence relations
of the Genocchi numbers as follows:%
\begin{equation}
G_{0}=0\text{, }(G+1)^{n}+G_{n}=\left\{ 
\begin{array}{c}
2,\text{ }n=1 \\ 
0,\text{ }n\neq 1%
\end{array}%
\right.  \label{2bB}
\end{equation}%
where $G^{n}$ is replaced by $G_{n}$.

Relations between Genocchi numbers, Euler numbers and Bernoulli numbers are
given by:%
\begin{equation}
E_{n}=\frac{G_{n+1}}{n+1},  \label{2b}
\end{equation}

\begin{eqnarray*}
G_{2n} &=&2(1-2^{2n})B_{2n} \\
&=&2nE_{2n-1}(0),
\end{eqnarray*}%
where $B_{n}$ and $E_{n}$ are Bernoulli numbers and Euler Numbers
respectively, cf. (\cite{jangGenoc}, \cite{kimArix}, \cite{kimArXiv}, \cite%
{Kim2007b}, \cite{kimRimJMAA2007}, \cite{TKimASCM2008Genochi}, \cite%
{RimParkMoonAAA}, \cite{ysimsekKurtNaciKim}, \cite{ysimsekNA}).

The ordinary Genocchi Polynomials are defined by means of the following
generating function:%
\begin{equation}
\mathrm{f}_{G}(t,x)=f_{G}(t)e^{xt}=\sum_{n=0}^{\infty }G_{n}(x)\frac{t^{n}}{%
n!}.  \label{1.3}
\end{equation}

From (\ref{1}) and (\ref{1.3}), we easily see that 
\begin{equation*}
G_{n}(x)=\sum_{k=0}^{n}\binom{n}{k}G_{k}x^{n-k}.
\end{equation*}

Observe that $G_{0}=0$, $G_{1}=1$, $G_{3}=G_{5}=G_{7}=\cdots =G_{2n+1}=0$, $%
n\in \mathbb{Z}^{+}$ cf. (\cite{jangGenoc}, \cite{jangkim}, \cite%
{JangKimJIA2009}, \cite{kimArix}, \cite{TKimASCM2008Genochi}, \cite%
{TKimLCJangHKPak}).

In \cite{LuoquDebna} and \cite{Luo-Qi-Debnath}, Luo et al defined new type
generalized Bernoulli polynomials and Euler polynomials depending on three
positive arbitrary real parameters. They proved many identities and
relations related to these polynomials. Main motivation of the work is to
define generating functions of Genocchi type numbers and polynomials
depending on three positive arbitrary real parameters.

Luo et al (\cite{LuoquDebna}, \cite{Luo-Qi-Debnath}) did not define
interpolation functions of their numbers and polynomials. On the other hand,
in this present paper, we can construct interpolation functions of our new
numbers and polynomials which depending on three positive arbitrary real
parameters. We also prove some new relations and properties associated with
these numbers, polynomials and intepolation functions.

We now summarize our paper as follows:

In Section 2, we construct generating functions of the Genocchi type numbers
and polynomials. We give reoccurrence relations of these numbers. We prove
multiplication theorem of these polynomials. We also give some properties of
these numbers and polynomials.

In Section 3, we find formula of the alternating sums of powers of
consecutive integers.

In section 4, by using derivative operator $\frac{d^{k}f(t)}{dt^{k}}%
\left\vert _{t=0}\right. $ to the generating function of the Genocchi type
numbers and polynomials, we construct interpolation functions of these
numbers and polynomials.

In Section 5, we give further remarks and observations on interpolation
functions.

\section{Genocchi Type Numbers and Polynomials}

In this section, by using same method that of , we define generalized
Genocchi number and polynomial depending on three positive arbitrary real
parameters. We investigate fundamental properties of these numbers and
polynomials.

Throughout of this paper $a$, $b$ and $c$ are positive real parameters with $%
a\neq b$ and $x\in 
%TCIMACRO{\U{211d} }%
%BeginExpansion
\mathbb{R}
%EndExpansion
$.

Now we are ready to define generating function of Genocchi type number
depending on three positive arbitrary real parameters as follows:

\begin{equation}
F(t;a,b)=\frac{2t}{b^{t}+a^{t}}=\sum_{n=0}^{\infty }\mathcal{G}_{n}(a,b)%
\frac{t^{n}}{n!}\text{, }\left\vert t\right\vert <\frac{\pi }{\left\vert \ln
a-\ln b\right\vert }.  \label{2.1a}
\end{equation}

By using (\ref{2.1a}) and the \textit{umbral calculus }convention, we obtain%
\begin{equation*}
\frac{2te^{-t\ln a}}{e^{t(\ln b-\ln a)}+1}=e^{\mathcal{G}_{n}(a,b)t}
\end{equation*}%
After some calculations, we get the following reoccurrence relations for the
number $\mathcal{G}_{n}(a,b)$ as follows:

Let $\mathcal{G}_{0}(a,b)=0$ and $\mathcal{G}_{1}(a,b)=1$. For $n\geq 2$,%
\begin{equation}
\mathcal{G}_{n}(a,b)+\dsum\limits_{k=0}^{n}\binom{n}{k}(\ln b-\ln a)^{k-n}%
\mathcal{G}_{k}(a,b)=2n\ln ^{n-1}\left( \frac{1}{a}\right) .  \label{2bBb}
\end{equation}

By using (\ref{2bBb}), we give few Genocchi-type numbers as follows:%
\begin{eqnarray*}
\mathcal{G}_{2}(a,b) &=&-\ln a-\ln b, \\
\mathcal{G}_{3}(a,b) &=&-6\ln ^{2}a+3\ln a\ln b.
\end{eqnarray*}

\begin{remark}
By substituting $a=1$, $b=e$ into (\ref{2.1a}) and (\ref{2bB}), then we
arrive at (\ref{1}) and (\ref{2bBb}), respectively. That is, the number $%
\mathcal{G}_{n}(1,e)$ reduces to the number $G_{n}$.
\end{remark}

\begin{lemma}
Let $a$, $b$ be arbitrary positive real parameters. Then we have%
\begin{equation}
\mathcal{G}_{n}(a,b)=(\ln b-\ln a)^{n-1}G_{n}\left( \frac{\ln a}{\ln a-\ln b}%
\right) \text{,}  \label{2.2}
\end{equation}
\end{lemma}

and%
\begin{equation}
\mathcal{G}_{n}(a,b)=\sum_{k=0}^{n}\binom{n}{k}(-1)^{n-k}(\ln a)^{n-k}(\ln
b-\ln a)^{k-1}G_{k}.  \label{2.3}
\end{equation}

\begin{proof}
We firstly give proof of (\ref{2.2}). From (\ref{2.1a}), we have%
\begin{equation*}
\sum_{n=0}^{\infty }\mathcal{G}_{n}(a,b)\frac{t^{n}}{n!}=\sum_{n=0}^{\infty
}(\ln b-\ln a)^{n-1}G_{n}\left( \frac{\ln a}{\ln a-\ln b}\right) \frac{t^{n}%
}{n!}.
\end{equation*}%
By comparing the coefficient $\frac{z^{n}}{n!}$ in the both sides of the
above equation, we easily arrive at (\ref{2.2}). Secondly, we give proof of (%
\ref{2.3}). By using (\ref{2.1a}), we have%
\begin{eqnarray*}
&&\sum_{n=0}^{\infty }\mathcal{G}_{n}(a,b)\frac{t^{n}}{n!} \\
&=&\frac{1}{(\ln b-\ln a)}\sum_{n=0}^{\infty }G_{n}(\ln b-\ln a)^{n}\frac{%
t^{n}}{n!}\sum_{n=0}^{\infty }(-\ln a)^{n}\frac{t^{n}}{n!}.
\end{eqnarray*}%
By using Cauchy product in the above, we obtain%
\begin{eqnarray*}
&&\sum_{n=0}^{\infty }\mathcal{G}_{n}(a,b)\frac{t^{n}}{n!} \\
&=&\frac{1}{\ln b-\ln a}\sum_{n=0}^{\infty }\left( \sum_{k=0}^{n}\binom{n}{k}%
G_{k}(-1)^{n-k}(\ln a)^{n-k}(\ln b-\ln a)^{k}\right) \frac{t^{n}}{n!}.
\end{eqnarray*}%
By comparing the coefficient $\frac{t^{n}}{n!}$ in the both sides of the
above equation, we easily arrive at (\ref{2.3}).
\end{proof}

Genocchi type polynomials, depending on three positive arbitrary real
parameters, are defined by means of the following generating function:

Let%
\begin{equation}
\mathcal{F}(t,x;a,b,c)=F(t;a,b)c^{xt}=\sum_{n=0}^{\infty }\mathcal{G}%
_{n}(x;a,b,c)\frac{t^{n}}{n!}\text{,}  \label{2.4}
\end{equation}%
where $\left\vert t\right\vert <\frac{\pi }{\left\vert \ln a-\ln
b\right\vert }$.

\begin{remark}
If $x=0$, then (\ref{2.4}) reduces to (\ref{2.1a}). By substituting $a=1$, $%
b=c=e$ into (\ref{2.4}), then we have%
\begin{eqnarray*}
\mathcal{F}(t,x;1,e,e) &=&\mathrm{f}_{G}(t,x). \\
F(t;1,e) &=&f_{G}(t).
\end{eqnarray*}%
From the above, we have 
\begin{equation*}
\mathcal{G}_{n}(x;1,e,e)=G_{n}(x),
\end{equation*}%
\begin{equation*}
\mathcal{G}_{n}(x;a,b,1)=\mathcal{G}(a,b),
\end{equation*}%
\begin{equation*}
\mathcal{G}_{n}(0;a,b,c)=\mathcal{G}_{n}(a,b),
\end{equation*}%
and%
\begin{equation*}
\mathcal{G}_{n}(0;1,e,e)=G_{n}.
\end{equation*}
\end{remark}

\begin{remark}
Recently, many kind generating functions related to Bernoulli, Euler \ and
Genocchi type polynomials have been found. Srivastava et al. \cite[pp. 254,
Eq. (20)]{SrivastawaGargeSC}, introduced and investigated the new type
generalization of the Bernoulli polynomials order $\alpha $, $\mathfrak{B}%
_{n}^{(\alpha )}(x;\lambda ;a,b,c)$, which are defined by means of the
following generating functions:%
\begin{equation}
\left( \frac{t}{\lambda b^{t}-a^{t}}\right) ^{\alpha
}c^{xt}=\sum_{n=0}^{\infty }\mathfrak{B}_{n}^{(\alpha )}(x;\lambda ;a,b,c)%
\frac{t^{n}}{n!}\text{, }\left( \left\vert t\ln (\frac{a}{b})+\ln \lambda
\right\vert <2\pi ;1^{\alpha }:=1;x\in 
%TCIMACRO{\U{211d} }%
%BeginExpansion
\mathbb{R}
%EndExpansion
\right) .  \label{1S}
\end{equation}%
If we set $\alpha =1$ and $\lambda =-1$ in (\ref{1S}), then we have%
\begin{equation*}
\mathfrak{B}_{n}^{(1)}(x;-1;a,b,c)=-\frac{1}{2}\mathcal{G}_{n}(x;a,b,c).
\end{equation*}%
The numbers $\mathfrak{B}_{n}^{(1)}(x;-1;a,b,c)$ are related to
Apostol-Bernoulli numbers. Ozden et al. \cite{Ozden Simsek Srivastava} have
unifed and extend the generating functions of the generalized Bernoulli
polynomials, the generalized Euler polynomials and the generalized Genocchi
polynomials associated with the positive real parameters $a$ and $b$ and the
complex parameter $\beta $. By applying the Mellin transformation to the
generating function of the unification of Bernoulli, Euler and Genocchi
polynomials, they defined a unification of the zeta functions.
\end{remark}

\begin{theorem}
\label{TheoremG(x)}Let $a$, $b$, $c$ be arbitrary positive real parameters.
Then we have%
\begin{equation}
\mathcal{G}_{n}(x;a,b,c)=\sum_{k=0}^{n}\binom{n}{k}(x\ln c)^{n-k}\mathcal{G}%
_{k}(a,b),  \label{2.5}
\end{equation}%
or%
\begin{equation}
\mathcal{G}_{n}(x;a,b,c)=\sum_{k=0}^{n}\binom{n}{k}(x\ln c)^{n-k}(\ln b-\ln
a)^{n-1}G_{k}\left( \frac{\ln a}{\ln a-\ln b}\right) .  \label{2.6}
\end{equation}
\end{theorem}

\begin{proof}[Proof of (\protect\ref{2.5})]
By (7), we have%
\begin{equation*}
\sum_{n=0}^{\infty }\mathcal{G}_{n}(x;a,b,c)\frac{t^{n}}{n!}%
=\sum_{n=0}^{\infty }\mathcal{G}_{n}(a,b)\frac{t^{n}}{n!}\sum_{n=0}^{\infty
}(x\ln c)^{n}\frac{t^{n}}{n!}.
\end{equation*}%
By Cauchy product in the above, we easily see that%
\begin{equation*}
\sum_{n=0}^{\infty }\mathcal{G}_{n}(x;a,b,c)\frac{t^{n}}{n!}%
=\sum_{n=0}^{\infty }\left( \sum_{k=0}^{n}\binom{n}{k}\mathcal{G}%
_{k}(a,b)x^{n-k}(\ln c)^{n-k}\right) \frac{t^{n}}{n!}.
\end{equation*}%
By comparing the coefficient $\frac{z^{n}}{n!}$ in the both sides of the
above equation, we easily arrive at the desire result. By substituting (\ref%
{2.2}) into (\ref{2.5}) and (\ref{2.3}) into (\ref{2.5}), after some
elementary calculations, we arrive at the proofs of (\ref{2.6}).
\end{proof}

By (\ref{2.6}), we easily obtain the following corollary.

\begin{corollary}
Let $a$, $b$, $c$ be arbitrary positive real parameters. Then we have%
\begin{eqnarray*}
\mathcal{G}_{n}(x;a,b,c) &=&\sum_{k=0}^{n}\sum_{j=0}^{k}\binom{n}{j,n-k,k-j}%
(-1)^{k-j}x^{n-k} \\
&&\times (\ln c)^{n-k}(\ln a)^{k-j}(\ln b-\ln a)^{n+j-k-1}G_{j},
\end{eqnarray*}%
where $G_{j}$ denotes classical Genocchi numbers and%
\begin{equation*}
\binom{n}{j,n-k,k-j}=\frac{n!}{j!(n-k)!(k-j)!}.
\end{equation*}
\end{corollary}

We now give application of Theorem \ref{TheoremG(x)} as follows:

\begin{equation*}
\mathcal{G}_{0}(x;a,b,c)=0,
\end{equation*}%
\begin{equation*}
\mathcal{G}_{1}(x;a,b,c)=1,
\end{equation*}%
\begin{equation*}
\mathcal{G}_{2}(x;a,b,c)=2x\ln c-\ln a-\ln b,
\end{equation*}

If we take $a=1$, and $b=c=e$ in the above, then we obtain ordinary Genocchi
polynomials as follows: 
\begin{equation*}
\mathcal{G}_{0}(x;1,e,e)=0,
\end{equation*}%
\begin{equation*}
\mathcal{G}_{1}(x;1,e,e)=1,
\end{equation*}%
\begin{equation*}
\mathcal{G}_{2}(x;1,e,e)=2x-1,
\end{equation*}%
\begin{equation*}
\mathcal{G}_{3}(x;1,e,e)=3(x^{2}-x).
\end{equation*}

\QTP{Body Math}
By using (\ref{2.4}), we have%
\begin{eqnarray*}
&&\sum_{n=0}^{\infty }\mathcal{G}_{n}(x+1;a,b,c)\frac{t^{n}}{n!}=\frac{%
2tc^{(x+1)t}}{b^{t}+a^{t}} \\
&=&2tc^{xt}+\frac{2tc^{xt}(c^{t}-a^{t}-b^{t})}{b^{t}+a^{t}}%
=2\sum_{n=0}^{\infty }\frac{(\ln c)^{n}x^{n}t^{n+1}}{n!} \\
&&+2\left( \sum_{n=0}^{\infty }\frac{\mathcal{G}_{n}(x;a,b,c)t^{n}}{n!}%
\right) \left( \sum_{n=0}^{\infty }\frac{\left( (\ln c)^{n}-(\ln a)^{n}-(\ln
b)^{n}\right) t^{n}}{n!}\right) .
\end{eqnarray*}%
After some elementary calculations in the above, we obtain%
\begin{eqnarray}
&&\sum_{n=0}^{\infty }\mathcal{G}_{n}(x+1;a,b,c)\frac{t^{n}}{n!}  \notag \\
&=&-2\mathcal{G}_{0}+t\left( 2+2(\ln c-\ln a-\ln b)\mathcal{G}_{0}-2\mathcal{%
G}_{1}(x;a,b,c)\right)  \notag \\
&&+\sum_{n=2}^{\infty }\left( 2n(\ln c)^{n-1}x^{n-1}-\mathcal{G}%
_{n}(x;a,b,c)\right) \frac{t^{n}}{n!}  \notag \\
&&+\sum_{n=2}^{\infty }\left( \sum_{l=0}^{n-1}\binom{n}{l}(\ln c)^{n-l}-(\ln
a)^{n-l}-(\ln b)^{n-l}\right) \mathcal{G}_{n}(x;a,b,c)\frac{t^{n}}{n!}.
\label{2.10}
\end{eqnarray}

Thus, by using the above equation, we arrive at the following corollary:

\begin{corollary}
Let $a$, $b$, $c$ be arbitrary positive real parameters. Then we have%
\begin{equation}
\mathcal{G}_{n}(x+1;a,b,c)=\sum_{k=0}^{n}\binom{n}{k}\mathcal{G}%
_{k}(x;a,b,c)(\ln c)^{n-k},  \label{2.8}
\end{equation}%
or%
\begin{equation}
\mathcal{G}_{n}(x+1;a,b,c)=\mathcal{G}_{n}(x;\frac{a}{c},\frac{b}{c},c).
\label{2.9}
\end{equation}
\end{corollary}

By substituting $a=1$, $b=c$ into (\ref{2.10}) with $n\geq 1$, we have%
\begin{equation*}
G_{n}(x+1;1,b,b)=2n(\ln b)^{n-1}x^{n-1}-G_{n}(x;1,b,b).
\end{equation*}%
By substituting $b=e$ into (\ref{2.10}) with $n\geq 1$, we obtain%
\begin{equation*}
G_{n}(x+1)=2nx^{n-1}-G_{n}(x).
\end{equation*}

By using (\ref{2.4}), we easily arrive at the following result:

\begin{corollary}
The Generalized Genocchi polynomial is satisfying following relations%
\begin{equation*}
\mathcal{G}_{n}(x+y;a,b,c)=\sum_{k=0}^{n}\binom{n}{k}\mathcal{G}%
_{k}(x;a,b,c)(\ln c)^{n-k}y^{n-k},
\end{equation*}%
or%
\begin{equation*}
\mathcal{G}_{n}(x+y;a,b,c)=\sum_{k=0}^{n}\binom{n}{k}\mathcal{G}%
_{k}(y;a,b,c)(\ln c)^{n-k}x^{n-k}.
\end{equation*}
\end{corollary}

\begin{theorem}
(Multiplication Theorem) Let $a$, $b$, $c$ be arbitrary positive real
parameters. Then we have%
\begin{equation*}
\dsum\limits_{n=0}^{\infty }\mathcal{G}_{n}(x;a,b,c)\frac{t^{n}}{n!}%
=y^{n-1}\dsum\limits_{j=0}^{y-1}(-1)^{j}\mathcal{G}_{n}\left( \frac{j}{y}%
;a,b,\frac{c^{\left( \frac{x}{y}\right) }b^{\left( \frac{j}{y}\right) }}{%
a^{\left( \frac{j+1}{y}\right) }}\right) .
\end{equation*}
\end{theorem}

\begin{proof}
By using (\ref{2.4}), we have%
\begin{equation*}
\dsum\limits_{n=0}^{\infty }\mathcal{G}_{n}(x;a,b,c)\frac{t^{n}}{n!}%
=y^{n-1}\dsum\limits_{n=0}^{\infty }\dsum\limits_{j=0}^{y-1}(-1)^{j}\mathcal{%
G}_{n}\left( \frac{j}{y};a,b,\frac{c^{\left( \frac{x}{y}\right) }b^{\left( 
\frac{j}{y}\right) }}{a^{\left( \frac{j+1}{y}\right) }}\right) \frac{t^{n}}{%
n!}.
\end{equation*}%
After some calculations in the above, we arrive at the desired result.
\end{proof}

We now define Genocchi type polynomial of higher order as follows:%
\begin{equation}
\mathcal{F}^{(k)}(t,x;a,b,c)=\left( \frac{2t}{b^{t}+a^{t}}\right)
^{k}c^{xt}=\sum_{n=0}^{\infty }\mathcal{G}_{n}^{(k)}(x;a,b,c)\frac{t^{n}}{n!}%
,  \label{2.13}
\end{equation}%
where $\mathcal{G}_{n}^{(k)}(x;a,b,c)$ denotes the Genocchi type polynomial
of higher order and $k$ is positive integer.

Observe that $\mathcal{F}^{(1)}(t,x;a,b,c)=\mathcal{F}(t,x;a,b,c)$ and $%
\mathcal{G}_{n}^{(1)}(x;a,b,c)=\mathcal{G}_{n}(x;a,b,c)$.

By using (\ref{2.13}), we obtain%
\begin{eqnarray*}
&&\sum_{n=0}^{\infty }\mathcal{G}_{n}^{(l+k)}(x+y;a,b,c)\frac{t^{n}}{n!} \\
&=&(\frac{2t}{a^{t}+b^{t}})^{l+k}c^{(x+y)t} \\
&=&\sum_{n=0}^{\infty }\mathcal{G}_{n}^{(l)}(x;a,b,c)\frac{t^{n}}{n!}%
\sum_{n=0}^{\infty }\mathcal{G}_{n}^{(k)}(y;a,b,c)\frac{t^{n}}{n!}.
\end{eqnarray*}%
From the above, we obtain%
\begin{eqnarray*}
&&\sum_{n=0}^{\infty }\mathcal{G}_{n}^{(l+k)}(x+y;a,b,c)\frac{t^{n}}{n!} \\
&=&\sum_{n=0}^{\infty }\left( \sum_{j=0}^{n}\binom{n}{j}\mathcal{G}%
_{j}^{(l)}(x;a,b,c)\mathcal{G}_{n-j}^{(k)}(y;a,b,c)\right) \frac{t^{n}}{n!}.
\end{eqnarray*}%
After some elementary calculations, we arrive at the following theorem:

\begin{theorem}
Let $l$ and $k$ be positive integers. Let $a$, $b$, $c$ be arbitrary
positive real parameters. Then we have%
\begin{equation}
\mathcal{G}_{n}^{(l+k)}(x+y;a,b,c)=\sum_{j=0}^{n}\binom{n}{j}\mathcal{G}%
_{j}^{(l)}(x;a,b,c)\mathcal{G}_{n-j}^{(k)}(y;a,b,c).  \label{2.14}
\end{equation}
\end{theorem}

\section{\textbf{The alternating sums of powers of consecutive integers}}

In 1713, J. Bernoulli discovered a formula for the sum $\sum%
\limits_{b=0}^{n}b^{j}$ for $j\in \mathbb{Z}^{+}$. In this section we prove
the alternating sums of powers of consecutive integers for $\mathcal{G}%
_{n}(x;a,b,c)$.

By using (\ref{2.1a}) and (\ref{2.4}), we obtain%
\begin{eqnarray*}
&&F(t;1,b)-(-1)^{m}\mathcal{F}(t,m;1,b,b) \\
&=&\sum_{n=0}^{\infty }\left( \mathcal{G}_{n}(1,b)-(-1)^{m}\mathcal{G}%
_{n}(m;1,b,b)\right) \frac{t^{n}}{n!}.
\end{eqnarray*}%
From the above, we have%
\begin{equation*}
\sum_{k=0}^{m-1}(-1)^{k}b^{kt}=\sum_{n=0}^{\infty }\left( \frac{\mathcal{G}%
_{n}(1,b)-(-1)^{m}\mathcal{G}_{n}(m;1,b,b)}{2}\right) \frac{t^{n-1}}{n!}.
\end{equation*}%
After some calculations, we arrive at the following theorem:

\begin{theorem}
\label{consecutive}Let $m$ and $n$ be positive integers. Let $b$ be
arbitrary positive real parameter. Then we have%
\begin{equation*}
\sum_{k=0}^{m-1}(-1)^{k}k^{n}=\frac{\mathcal{G}_{n}(1,b)-(-1)^{m}\mathcal{G}%
_{n}(m;1,b,b)}{2n\ln ^{n}\left( b\right) }.
\end{equation*}
\end{theorem}

\begin{remark}
By substituting $b=e$ into Theorem \ref{consecutive}, then we have%
\begin{eqnarray*}
\sum_{k=0}^{m-1}(-1)^{k}k^{n} &=&\frac{\mathcal{G}_{n}(1,e)-(-1)^{m}\mathcal{%
G}_{n}(m;1,e,e)}{2n} \\
&=&\frac{G_{n}-(-1)^{m}G_{n}(m)}{2n}.
\end{eqnarray*}%
Thus, by (\ref{2b}), Theorem \ref{consecutive} reduces to%
\begin{equation*}
\sum_{k=0}^{m-1}(-1)^{k}k^{n}=\frac{E_{n}-(-1)^{m}E_{n}(m)}{2},
\end{equation*}%
cf. (\cite{jangGenoc}, \cite{jangkim}, \cite{kimArXiv}, \cite{OzdenAAA}, 
\cite{Luo-Qi-Debnath}, \cite{LuoquDebna}, \cite{ysimsekKurtNaciKim}).
\end{remark}

\section{Interpolation Functions}

In this section, we construct interpolation function of the generalized
Genocchi type numbers and polynomials on $\mathbb{C}$.

By using (\ref{2.4}), we have%
\begin{eqnarray*}
\mathcal{F}(t,x;a,b,c) &=&2t\sum_{n=0}^{\infty }(-1)^{n}e^{(x\ln c-\ln
a+n\ln \frac{b}{a})t} \\
&=&\sum_{n=0}^{\infty }\mathcal{G}_{n}(x;a,b,c)\frac{t^{n}}{n!}.
\end{eqnarray*}%
By applying derivative operator $\frac{d^{k}\mathcal{F}(t,x;a,b,c)}{dt^{k}}%
\left\vert _{t=0}\right. $ to the above, we obtain%
\begin{equation*}
\mathcal{G}_{k}(x;a,b,c)=2k\sum_{n=0}^{\infty }(-1)^{n}(n\ln c-\ln a+n\ln 
\frac{b}{a})^{k-1}.
\end{equation*}%
Therefore, by using the above relation, we obtain the following theorem.

\begin{theorem}
\label{interpol}Let $k\in \mathbb{Z}^{+}$. Let $a$, $b$, $c$ be arbitrary
positive real parameters. Then we have%
\begin{equation*}
\frac{\mathcal{G}_{k}(x;a,b,c)}{k}=2\sum_{n=0}^{\infty }(-1)^{n}(x\ln c-\ln
a+n\ln \frac{b}{a})^{k-1}.
\end{equation*}
\end{theorem}

By Theorem \ref{interpol}, we can derive Hurwitz type generalized Genocchi
zeta function, which interpolates Genocchi polynomials at negative integers,
as follows.

\begin{definition}
\label{Defnterpola}Let $s\in \mathbb{C}$. Let $a$, $b$, $c$ be arbitrary
positive real parameters. We define%
\begin{equation*}
\mathfrak{Z}_{\mathcal{G}}(s,x;a,b,c)=2\sum_{n=0}^{\infty }\frac{(-1)^{n}}{%
(x\ln c-\ln a+n\ln \frac{b}{a})^{s}}.
\end{equation*}
\end{definition}

From Definition \ref{Defnterpola}, we see that%
\begin{equation}
\mathcal{Z}_{\mathcal{G}}(s,1;a,b,1)=2\sum_{n=0}^{\infty }\frac{(-1)^{n}}{%
(-\ln a+n\ln \frac{b}{a})^{s}}.  \label{1bu}
\end{equation}%
By substituting $s=-n$, with $n\in \mathbb{Z}^{+}$, into Definition \ref%
{Defnterpola} and using Theorem \ref{interpol}, we arrive at the following
Theorem.

\begin{theorem}
\label{interpoResult}Let $n\in \mathbb{Z}^{+}$. Let $a$, $b$, $c$ be
arbitrary positive real parameters. Then we have%
\begin{equation*}
\mathfrak{Z}_{\mathcal{G}}(-n,x;a,b,c)=\frac{\mathcal{G}_{n}(x;a,b,c)}{n}.
\end{equation*}
\end{theorem}

Observe that setting $x=1$ in Theorem \ref{interpoResult}, we get
interpolation function of the numbers $\mathcal{G}_{n}(a,b)$ as follows:%
\begin{equation*}
\mathcal{Z}_{\mathcal{G}}(-n,1;a,b,1)=\frac{\mathcal{G}_{n}(a,b)}{n}.
\end{equation*}

By setting $n=j+my$ with $j=1,2,\cdots ,y$, $y$ is an odd integer, and $%
m=0,1,\cdots ,\infty $ in (\ref{1bu}), we obtain%
\begin{equation*}
\mathcal{Z}_{\mathcal{G}}(s,1;a,b,1)=\frac{1}{y^{s}}\sum_{j=1}^{y}(-1)^{j}%
\sum_{m=0}^{\infty }\frac{(-1)^{m}}{\left( -\frac{\ln a}{y}+\frac{j\ln \frac{%
b}{a}}{y}+m\ln \frac{b}{a}\right) ^{s}}.
\end{equation*}%
After some calculations in the above, we arrive at the following corollary:

\begin{corollary}
Let $y$ be an odd integer. Let $a$, $b$, $c$ be arbitrary positive real
parameters. Then we have%
\begin{equation*}
\mathcal{Z}_{\mathcal{G}}(s,1;a,b,1)=\frac{1}{y^{s}}\sum_{j=1}^{y}(-1)^{j}%
\mathfrak{Z}_{\mathcal{G}}\left( s,1;a,b,\frac{b^{\left( \frac{j}{y}\right) }%
}{a^{\left( \frac{y+j-1}{y}\right) }}\right) .
\end{equation*}
\end{corollary}

\section{Further Remarks and Observations}

The function $\mathfrak{Z}_{\mathcal{G}}(s,x;a,b,c)$ and $\mathcal{Z}_{%
\mathcal{G}}(s,1;a,b,1)$\ are related to the \textit{Lerch trancendent} $%
\Phi (z,s,a)$ which is the analytic continuation of the series%
\begin{eqnarray*}
\Phi (z,s,a) &=&\frac{1}{a^{s}}+\frac{z}{(a+1)^{s}}+\frac{z}{(a+2)^{s}}%
+\cdots q \\
&=&\sum_{n=0}^{\infty }\frac{z^{n}}{(n+a)^{s}},
\end{eqnarray*}%
cf. see \cite[p. 121 et seq.]{SrivastavaChoi}, \cite{Guillera and Sondow}.
The series $\sum_{n=0}^{\infty }\frac{z^{n}}{(n+a)^{s}}$\ converge for $a\in 
\mathbb{C\diagdown Z}_{0}^{-}$, $s\in \mathbb{C}$ when $\left\vert
z\right\vert <1$; $\Re (s)>1$ when $\left\vert z\right\vert =1$ where%
\begin{equation*}
\mathbb{Z}_{0}^{-}=\mathbb{Z}^{-}\cup \left\{ 0\right\} ,\ \mathbb{Z}%
^{-}=\left\{ -1,-2,-3,...\right\} .
\end{equation*}%
\ \ $\Phi $ denotes the familiar Hurwitz-Lerch Zeta function (cf. \cite[p.
121 et seq.]{SrivastavaChoi}, \cite{Guillera and Sondow}, \cite{SimsekArxiv}%
, \cite{kimArix},\cite{TKimLCJangHKPak}).

The Lerch zeta function $\Phi (z,s,u)$ is the analytic continuation of the
following series%
\begin{equation}
\dsum\limits_{n=0}^{\infty }\frac{z^{n}}{(n+u)^{s}}  \label{A1}
\end{equation}%
which converge for any real number $u>0$ if $z$ and $s$ are any complex
numbers with either $\left\vert z\right\vert <1$, or $\left\vert
z\right\vert =1$ and $\Re (s)>1$.

The functio $\Phi (z,s,u)$ is related to many special functions, some of
them\ are given as follows, cf. (\cite[p. 124]{SrivastavaChoi}, \cite%
{Guillera and Sondow}, \cite{SimsekArxiv}):

Special cases include the analytic continuations of the Riemann zeta function%
\begin{equation*}
\Phi (1,s,1)=\zeta (s)=\sum_{n=1}^{\infty }\frac{1}{n^{s}}\text{, }\Re (s)>1,
\end{equation*}%
the Hurwitz zeta function%
\begin{equation*}
\Phi (1,s,a)=\zeta (s,a)=\sum_{n=0}^{\infty }\frac{1}{(n+a)^{s}}\text{, }\Re
(s)>1,
\end{equation*}%
the alternating zeta function (also called Dirichlet's eta function $\eta
(s) $)%
\begin{equation*}
\Phi (-1,s,1)=\zeta ^{\ast }(s)=\sum_{n=1}^{\infty }\frac{(-1)^{n-1}}{n^{s}},
\end{equation*}%
the Dirichlet beta function%
\begin{equation*}
\frac{\Phi (-1,s,\frac{1}{2})}{2^{s}}=\beta (s)=\sum_{n=0}^{\infty }\frac{%
(-1)^{n}}{(2n+1)^{s}},
\end{equation*}%
the Legendre chi function%
\begin{equation*}
\frac{z\Phi (z^{2},s,\frac{1}{2})}{2^{s}}=\chi _{s}(z)=\sum_{n=0}^{\infty }%
\frac{z^{2n+1}}{(2n+1)^{s}}\text{, }\left\vert z\right\vert \leq 1;\Re (s)>1%
\text{,}
\end{equation*}%
the polylogarithm%
\begin{equation*}
z\Phi (z,n,1)=Li_{m}(z)=\sum_{n=0}^{\infty }\frac{z^{k}}{n^{m}}
\end{equation*}%
and the Lerch zeta function (sometimes called the Hurwitz-Lerch zeta
function)%
\begin{equation*}
L(\lambda ,\alpha ,s)=\Phi (e^{2\pi i\lambda },s,\alpha ),
\end{equation*}%
which is a special function and generalizes the Hurwitz zeta function and
polylogarithm, cf. (\cite{Guillera and Sondow}, \cite{SrivastavaChoi}, \cite%
{jangGenoc}, \cite{JangKimJIA2009}, \cite{Kim2007b}, \cite{kimJIAgeno2008}, 
\cite{kimjnmp2007}, \cite{TKimASCM2008Note}, \cite{TkimMSkimASCM}, \cite%
{SimsekArxiv}) and see also the references cited in each of these earlier
works.

Setting $a=1,b=c=e$ in Definition \ref{Defnterpola}, then we have 
\begin{equation*}
\mathfrak{Z}_{\mathcal{G}}(s,x;1,e,e)=2\sum_{n=0}^{\infty }\frac{(-1)^{n}}{%
(x+n)^{s}},
\end{equation*}%
and%
\begin{equation*}
\mathcal{Z}_{\mathcal{G}}(s,1;1,e,e)=2\sum_{n=0}^{\infty }\frac{(-1)^{n}}{%
n^{s}}.
\end{equation*}

By using (\ref{A1}), the function $\mathfrak{Z}_{\mathcal{G}}(s,x;1,e,e)$
and $\mathcal{Z}_{\mathcal{G}}(s,1;1,e,e)$\ satisfies the following
identities:%
\begin{equation*}
\mathfrak{Z}_{\mathcal{G}}(s,x;1,e,e)=-2\Phi (-1,s,x).
\end{equation*}%
and%
\begin{eqnarray*}
\mathcal{Z}_{\mathcal{G}}(s,1;1,e,e) &=&2\Phi (-1,s,1) \\
&=&-2\zeta ^{\ast }(s).
\end{eqnarray*}

\begin{acknowledgement}
This paper was supported by the Scientific Research Fund of Project
Administration of Akdeniz University.
\end{acknowledgement}

\end{document}